\documentclass{amsart}
\usepackage{amsmath, amssymb,mathtools,url,color,bm}

\theoremstyle{theorem}
\newtheorem{theorem}{Theorem}[section]
\newtheorem{lemma}[theorem]{Lemma}

\newtheorem{cor}[theorem]{Corollary}

\theoremstyle{definition}

\newtheorem{example}[theorem]{Example}

\theoremstyle{remark}
\newtheorem{remark}[theorem]{Remark}

\newcommand{\F}{\ensuremath{\mathbb F}}

\newcommand{\Q}{\ensuremath{\mathbb Q}}
\newcommand{\Z}{\ensuremath{\mathbb Z}}
\newcommand{\I}{\ensuremath{\mathcal I}}

\DeclareMathOperator{\ord}{ord}
\DeclareMathOperator{\lcm}{lcm}

\numberwithin{equation}{section}

\begin{document}

\title[Theorems of Chevalley-Warning and Ax-Katz]{A generalization of the theorems of Chevalley-Warning and Ax-Katz\\ via polynomial substitutions}

\author{Ioulia N. Baoulina} 
\address{Department of Mathematics, Moscow State Pedagogical University, Krasnoprudnaya str. 14, Moscow 107140, Russia}
\email{jbaulina@mail.ru}
\author{Anurag Bishnoi}
\address{Freie Universit\"at Berlin, Institut f\"ur Mathematik, Arnimallee 3, 14195 Berlin, Germany.}
\email{anurag.2357@gmail.com}
\author{Pete L. Clark}
\address{Department of Mathematics, 1023 D. W. Brooks Drive, Athens, GA 30605, USA}
\email{plclark@gmail.com}


\date{\today}

\commby{}

\begin{abstract}
We give conditions under which the number of solutions of a system of polynomial equations over a finite field $\F_q$ of characteristic~$p$ is divisible by $p$.  Our setup involves the substitution $t_i \mapsto f(t_i)$ for auxiliary polynomials $f_1,\dots,f_n \in \F_q[t]$.  We recover as special cases results of Chevalley-Warning and Morlaye-Joly.  Then we investigate higher $p$-adic divisibilities, proving a 
result that recovers the Ax-Katz Theorem.  We also consider $p$-weight degrees, recovering work of Moreno-Moreno, Moreno-Castro and Castro-Castro-Velez.
\end{abstract}

\maketitle

\section{Introduction}

We denote the positive integers by $\Z^+$ and the nonnegative integers by $\Z^{\geq 0}$.  We make the standard combinatorial convention that $0^0 = 1$.   Let $\F_q$ be a finite field of order $q = p^s$ and let $\F_q^{\times}=\F_q \setminus \{0\}$.

\subsection{Generalizing the Chevalley-Warning theorem}
We begin by recalling the following classical result.

\begin{theorem}[Chevalley-Warning \cite{Chevalley35}, \cite{Warning35}]
\label{CHEVWARNINGTHM}
Let $P_1,\dots,P_r \in \F_q[t_1,\dots,t_n]$ be nonzero polynomials.  Suppose that $\sum_{j=1}^r \deg(P_j) < n$.  Then
\[p \mid \#  \{ (y_1,\dots,y_n) \in \F_q^n \mid P_j(y_1,\dots,y_n) = 0 \text{ for all } 1 \leq j \leq r \}. \]
\end{theorem}

A polynomial $f \in \F_q[t]$ is a \textbf{permutation polynomial} if the associated evaluation map $E(f): \F_q \rightarrow \F_q$ given by $x \mapsto f(x)$ is a bijection.  The following result is an immediate consequence of Theorem~\ref{CHEVWARNINGTHM}.

\begin{cor}
Let $P_1,\dots,P_r \in \F_q[t_1,\dots,t_n]$ be nonzero polynomials, and let $f_1,\dots,f_n \in \F_q[t]$ be permutation polynomials.  Suppose that $\sum_{j=1}^r \deg(P_j) < n$.  
Then $p \mid \# \{ (x_1,\dots,x_n) \in \F_q^n \mid P_j(f_1(x_1),\dots,f_n(x_n)) = 0 \text{ for all } 1 \leq j \leq r \}$. 
\end{cor}

In this paper we will give a sufficient condition for $p$ to divide
\[ \# \{ (x_1,\dots,x_n) \in \F_q^n \mid P_j(f_1(x_1),\dots,f_n(x_n)) = 0 \text{ for all } 1 \leq j \leq r\}, \]
where $f_1,\dots,f_n \in \F_q[t]$ are \emph{any} polynomials.  For $f \in \F_q[t]$, let $u(f)$ 
be the least $\delta \in \Z^+$ such that $\sum_{x \in \F_q} f(x)^{\delta} \neq 0$ if such a $\delta$ exists; otherwise let 
$u(f) = \infty$.  Wan, Shiue and Chen~\cite{WSC93} showed that $u(f) < \infty$ implies $u(f) \leq \# f(\F_q)-1$.  From this and 
the standard fact that for $m\in\Z^+$ 
\begin{equation}
\label{THEFACT} \sum_{x \in \F_q} x^m = \begin{cases} -1 & \text{ if } (q-1) \mid m, \\ 0 & \text{ otherwise,} \end{cases}
\end{equation}
we see that $u(f) = q-1$ iff $f$ is a permutation polynomial.    We say that $f$ is a 
\textbf{WSC polynomial} if $u(f) = \# f(\F_q)-1$ and that $f$ is a \textbf{weakly WSC polynomial} if $u(f) < \infty$.  

\begin{example}
\label{MONOEX}
Let $m \in \Z^+$ and $f(t) = t^m$.  Put $d = \gcd(m,q-1)$.  Then $f(\F_q) = \{x^m \mid x \in \F_q\}$  $= \{x^d \mid x \in \F_q\}$,
so $\# f(\F_q) = \frac{q-1}{d} + 1$.  Using \eqref{THEFACT} we get $\sum_{x \in \F_q} f(x)^{\delta} = \sum_{x \in \F_q} x^{m \delta} \neq 0$ iff $(q-1) \mid m \delta$, so 
$m\cdot u(f) = \lcm(m,q-1)$ and $u(f) = (q-1)/d$.  Thus $f$ is a WSC polynomial. 
\end{example}

Let $\I \subseteq \{1,\dots,n\}$ be nonempty.  For a monomial $a t_1^{m_1} \cdots t_n^{m_n}$ 
with $a \in \F_q^{\times}$, we define the $\I$-degree $\deg_{\I}(a t_1^{m_1} \cdots t_n^{m_n}) \coloneqq \sum_{i \in \I} m_i$.  For $P \in \F_q[t_1,\dots,t_n]$, we define the $\I$-degree $\deg_{\I}(P)$ to be the maximum of the $\I$-degrees of its monomial terms (and $\deg_{\I}(0) = -\infty$). 

Here is our first main result.

\begin{theorem}
\label{MAINTHM}
Let $P_1,\dots,P_r \in \F_q[t_1,\dots,t_n]$ be nonzero polynomials, and let $f_1,\dots,f_n \in \F_q[t]$ be any polynomials. Let $\I\subseteq\{1,\dots,n\}$ be a nonempty subset. Suppose that
\begin{equation}
\label{MAINTHMEQ1}
(q-1) \sum_{j=1}^r \deg_{\I}(P_j) < \sum_{i\in\I} u(f_i).  
\end{equation}
Then $p\mid \# \{ (x_1,\dots,x_n) \in \F_q^n \mid P_j(f_1(x_1),\ldots,f_n(x_n)) = 0 \text{ for all } 1 \leq j \leq r\}$.
\end{theorem}

\begin{remark}
\label{NONWEAKREMARK}
A polynomial $f \in \F_q[t]$ is \emph{not} weakly WSC iff every fiber of the map $x \mapsto f(x)$ has size a multiple of $p$ \cite[Remark~2.3]{WSC93}, \cite[\S2.3]{Turnwald95}.  It follows that the conclusion of Theorem~\ref{MAINTHM} holds whenever some $f_i$ is not 
weakly WSC, without any hypotheses on the $\I$-degrees of the polynomials $P_j$.  
\end{remark}

Turnwald showed \cite[Proposition~2.3(d)]{Turnwald95} that for nonconstant $f \in \F_q[t]$, we have 
$u(f) \geq \frac{q-1}{\deg(f)}$ and thus also $u(f) \geq \bigl\lceil \frac{q-1}{\deg(f)} \bigr\rceil$.  So Theorem~\ref{MAINTHM} implies:

\begin{cor}
\label{COR2MAINTHM}
Let $P_1,\dots,P_r \in \F_q[t_1,\dots,t_n]$ be nonzero polynomials, and let\linebreak $\I\subseteq\{1,\dots,n\}$ be nonempty. Let $f_1,\dots,f_n \in \F_q[t]$ and assume that $f_i$ is nonconstant for all $i\in\I$. Suppose that
$(q-1)\sum_{j=1}^r \deg_{\I}(P_j) < \sum_{i\in\I} \bigl\lceil \frac{q-1}{\deg(f_i)} \bigr\rceil$. Then 
$p\mid \# \{ (x_1,\dots,x_n) \in \F_q^n \mid P_j(f_1(x_1),\dots,f_n(x_n)) = 0 \text{ for all } 1 \leq j \leq r\}$.
\end{cor}

From Theorem~\ref{MAINTHM} and the definition of a WSC polynomial, we get:
\begin{cor}
\label{CORMAINTHMNEW}
Let $P_1,\dots,P_r \in \F_q[t_1,\dots,t_n]$ be nonzero polynomials, and let $\I\subseteq\{1,\dots,n\}$ be nonempty.   Let $f_1,\dots,f_n \in \F_q[t]$ with $f_i$ a WSC polynomial for all $i\in\I$. Suppose that
\begin{equation}
\label{MAINTHMEQ}
(q-1) \sum_{j=1}^r \deg_{\I}(P_j) < \sum_{i\in\I} (\# f_i(\F_q)-1).
\end{equation}
Then $p\mid \# \{ (x_1,\dots,x_n) \in \F_q^n \mid P_j(f_1(x_1),\dots,f_n(x_n)) = 0 \text{ for all } 1 \leq j \leq r\}$.
\end{cor}

From Corollary~\ref{CORMAINTHMNEW} and Example~\ref{MONOEX} we get:

\begin{cor}
\label{CLASSICALMAINTHM}
Let $P_1,\dots,P_r \in \F_q[t_1,\dots,t_n]$ be nonzero polynomials, and let\linebreak $\I\subseteq\{1,\dots,n\}$ be nonempty. Let $f_1,\dots,f_n \in \F_q[t]$ and assume that $f_i(t)=t^{m_i}$ for $m_i \in \Z^+$ for all $i\in\I$. For $i\in\I$, put $d_i=\gcd(m_i,q-1)$. Suppose that 
\[
\sum_{j=1}^r \deg_{\I}(P_j) < \sum_{i\in\I} \frac 1{d_i}.
\]
Then $p\mid \# \{ (x_1,\dots,x_n) \in \F_q^n \mid P_j(f_1(x_1),\dots,f_n(x_n)) = 0 \text{ for all } 1 \leq j \leq r\}$.
\end{cor}

Taking $\I=\{1,\dots,n\}$, $r = 1$ and $\deg(P_1) = 1$ in Corollary~\ref{CLASSICALMAINTHM} we recover:

\begin{theorem}[Morlaye \cite{Morlaye71}, Joly \cite{Joly71}]
\label{MORLAYEJOLYTHM}
Let $a_1,\dots,a_n \in \F_q^{\times}$, $b \in \F_q$, and $m_1,\dots,m_n \in \Z^+$.  For $1\le i\le n$, put $d_i=\gcd(m_i,q-1)$. Suppose that $\sum_{i=1}^n (1/d_i) > 1$. 
Then $p\mid \# \{(x_1,\dots,x_n) \in \F_q^n \mid a_1^{} x_1^{m_1} + \dots + a_n^{} x_n^{m_n} = b\}$.
\end{theorem}

\subsection{Higher $\bm{p}$-adic Divisibilities}
Under the assumptions of Theorem~\ref{CHEVWARNINGTHM},  Warning~\cite{Warning35} also proved that either 
\[
\{ (y_1,\dots,y_n) \in \F_q^n \mid P_j(y_1,\dots,y_n) = 0 \text{ for all } 1 \leq j \leq r\}=\varnothing
\]
or
\[
\#  \{ (y_1,\dots,y_n) \in \F_q^n \mid P_j(y_1,\dots,y_n) = 0 \text{ for all } 1 \leq j \leq r\}\geq q^{n- \sum_{j=1}^r \deg(P_j)}.
\]
Theorem~\ref{CHEVWARNINGTHM} and this second theorem of Warning raise the following questions:
\begin{itemize}
\item[(Q1)] Do we always have 
\[
q \mid \#  \{ (y_1,\dots,y_n) \in \F_q^n \mid P_j(y_1,\dots,y_n)=0 \text{ for all } 1 \leq j \leq r\} ?
\] 
\item[(Q2)] For fixed $n$, $r$ and $\deg(P_1),\dots,\deg(P_r)$, what is the largest power of $p$ that always divides $\#  \{ (y_1,\dots,y_n) \in \F_q^n \mid P_j(y_1,\dots,y_n) = 0 \text{ for all } 1 \leq j \leq r\}$?
\end{itemize}
Ax~\cite{Ax64} answered (Q1) and, when $r = 1$, (Q2).  Katz~\cite{Katz71} fully answered (Q2).

\begin{theorem}[Ax-Katz \cite{Ax64}, \cite{Katz71}]
\label{AXKATZTHM}
\textup{(a)}
Let $P_1,\ldots,P_r \in \F_q[t_1,\ldots,t_n]$ be polynomials of positive degree.  Suppose that $\sum_{j=1}^r \deg(P_j) < n$. Then 
\[
q^{\lceil (n-\sum_{j=1}^r \deg(P_j))/\max_{1 \leq j \leq r} \deg(P_j) \rceil}
\] 
divides $\# \{(y_1,\dots,y_n) \in \F_q^n \mid P_j(y_1,\dots,y_n) = 0 \text{ for all } 1 \leq j \leq r\}$.

\textup{(b)}
For all $n,r,m_1,\ldots,m_r \in \Z^+$ there are $P_1,\ldots,P_r \in \F_q[t_1,\ldots,t_n]$ with $\deg(P_j) = m_j$ for $1 \leq j \leq r$ such that
\begin{align*}
\ord_p&\bigl(\# \{(y_1,\dots,y_n) \in \F_q^n \mid P_j(y_1,\dots,y_n) = 0  \text{ for all } 1 \leq j \leq r\}\bigr)\\
&= s\biggl{\lceil} \frac{n-\sum_{j=1}^r \deg(P_j)}{\max_{1 \leq j \leq r} \deg(P_j)} \biggr{\rceil}, 
\end{align*}
where $q=p^s$.
\end{theorem}

In the setting of Theorem~\ref{MORLAYEJOLYTHM}, Joly conjectured that the analogue of (Q1) has an affirmative answer, i.e., that $q\mid \# \{(x_1,\dots,x_n) \in \F_q^n \mid a_1^{} x_1^{m_1} + \dots + a_n^{} x_n^{m_n} = b\}$.  The following result of Wan affirms Joly's conjecture and also addresses (Q2).

\begin{theorem}[Wan~\cite{Wan88}]  
\label{WAN88THM}
Let $a_1,\dots,a_n\! \in \F_q^{\times}$, $b \in\! \F_q$, and $m_1,\dots,m_n \in \Z^+$.  For $1\le i\le n$, put $d_i=\gcd(m_i,q-1)$. Then
$$
q^{\lceil \sum_{i=1}^n (1/d_i)-1\rceil}\mid \# \{(x_1,\dots,x_n) \in \F_q^n \mid a_1^{} x_1^{m_1} + \dots + a_n^{} x_n^{m_n} = b\}.
$$
\end{theorem}

It is natural to ask (Q1) and (Q2) in the setting of Theorem~\ref{MAINTHM}. 

\begin{example}
\label{EXAMPLE1}
Let $\mathcal{I}=\{1,\dots,n\}$, $q=p^s$ with $s\geq 2$, and let $n\geq r\geq 1$. Put $P_j(t_1,\dots,t_n)=t_j$ for $1\leq j\leq r-1$, and $P_r(t_1,\dots,t_n)=t_r\cdots t_n$. Put $f_i(t)=t$ for $1\leq i\leq n-1$, and put $f_n(t)=t^p-t$.  The associated map $E(f_n): \F_q \rightarrow \F_q$ is $\F_p$-linear 
with kernel $\F_p$, so all nonempty fibers of $E(f_n)$ have size $p$ and $u(f_n) = \infty$.  Thus the hypothesis \eqref{MAINTHMEQ1} of 
Theorem~\ref{MAINTHM} holds, yet
\begin{align*}
\# &\{ (x_1,\dots,x_n) \in \F_q^n \mid P_j(f_1(x_1),\dots,f_n(x_n)) = 0 \text{ for all } 1 \leq j \leq r\}\\
&=q^{n-r+1}-(q-1)^{n-r}\cdot \#\{x \in \F_q \mid f_n(x)\neq 0\}=q^{n-r+1}-(q-1)^{n-r}(q-p)\\
&\equiv  (-1)^{n-r}p \not\equiv 0\pmod{q}.
\end{align*}
\end{example}

\begin{example}
\label{EXAMPLE2}
Let $q=3^4$, $n=3$, $r=1$, $P_1(t_1,t_2,t_3)=t_1+t_2+t_3$, $f_1(t)=f_2(t)=t^3+t^2+1$, $f_3(t)=t^{13}+t^{11}+t$, $\I=\{1,2,3\}$. Then $\deg_{\I}(P_1)=1$, $u(f_1)=u(f_2)=40$, $u(f_3)=14$, and so $80=(q-1)\deg_{\I}(P_1)<u(f_1)+u(f_2)+u(f_3)=94$. However, $\#\{(x_1,x_2,x_3)\in\F_{3^4}^3
\mid P_1(f_1(x_1),f_2(x_2),f_3(x_3))=0\}=3^3\cdot 13\cdot 19\not\equiv 0\pmod{3^4}$.
\end{example}

However, switching to a different numerical invariant yields a positive answer to (Q1).  For nonconstant $f\in\F_q[t]$, write $f(t)=\sum_{\ell=1}^R b_{\ell}t^{m_{\ell}}$ with $b_{\ell}\in\F_q^{\times}$ and $m_{\ell} \in \Z^{\geq 0}$ (with all $m_\ell$'s pairwise distinct). Define 
\[
\omega(f)\coloneqq \min\biggl\{\sum_{\ell=1}^R \gamma_{\ell}\,\,\Bigl|\Bigr.\,\, 0\le\gamma_1,\dots,\gamma_R\le q-1\text{ and } \sum_{\ell=1}^R m_{\ell}\gamma_{\ell}\in (q-1)\Z^+\biggr\}.
\]
We observe that $1\le \omega(f)\le q-1$.  Now we can prove:

\begin{theorem}
\label{AXKATZGENERAL}
Let $P_1,\dots,P_r \in \F_q[t_1,\dots,t_n]$ be nonzero polynomials, and let\linebreak $\I\subseteq\{1,\dots,n\}$ be a nonempty subset such that $\max_{1\le j\le r}\deg_{\I}(P_j)>0$. Let $f_1,\dots,f_n \in \F_q[t]$ and assume that $f_i$ is nonconstant for all $i\in\I$.  Suppose that $(q-1)\sum_{j=1}^r \deg_{\I}(P_j) < \sum_{i\in\I} \omega(f_i)$.
Then  
\[
q^{\left\lceil\left(\sum_{i\in\I} (\omega(f_i)/(q-1))-\sum_{j=1}^r \deg_{\I}(P_j)\right)/\max_{1\le j\le r}\deg_{\I}(P_j)\right\rceil}
\] 
divides $\# \{ (x_1,\dots,x_n) \in \F_q^n \mid P_j(f_1(x_1),\dots,f_n(x_n)) = 0 \text{ for all } 1 \leq j \leq r\}$.
\end{theorem}

Using the fact that  $\sum_{\ell=1}^R m_{\ell}\gamma_{\ell}\in(q-1)\Z^+$ in the definition of $\omega(f)$ implies $\sum_{\ell=1}^R \gamma_{\ell}\ge (q-1)/\deg(f)$, together with \cite[Corollary~2.4]{ZC14}, we conclude
\begin{equation}
\label{AXKATZIMPLIES1}
(q-1)/\deg(f)\le\omega(f)\le u(f). 
\end{equation}
Thus the ``low degree'' hypothesis in Theorem~\ref{AXKATZGENERAL} is in general more stringent than in Theorem~\ref{MAINTHM} -- as must be the case in view of Examples~\ref{EXAMPLE1} and \ref{EXAMPLE2}.   
However, the conditions agree in an important case: by~\cite[Proposition~2.3(d)]{Turnwald95}, we have 
\begin{equation}
\label{AXKATZIMPLIES2}
\deg(f) \mid (q-1) \iff u(f)=\omega(f)=(q-1)/\deg(f).
\end{equation}
Via \eqref{AXKATZIMPLIES1} and \eqref{AXKATZIMPLIES2}, Theorem~\ref{AXKATZGENERAL} implies the following results.

\begin{cor}
\label{QDIVISIBILITY}
Let $P_1,\dots,P_r \in \F_q[t_1,\dots,t_n]$ be nonzero polynomials, and let $\I\subseteq\{1,\dots,n\}$ be  nonempty. Let $f_1,\dots,f_n \in \F_q[t]$ and suppose that $f_i$ is nonconstant with $\deg(f_i)\mid(q-1)$ for all $i\in\I$. Also suppose that
$(q-1) \sum_{j=1}^r \deg_{\I}(P_j) < \sum_{i\in\I} u(f_i)$. Then
\[q\mid \# \{ (x_1,\dots,x_n) \in \F_q^n \mid P_j(f_1(x_1),\dots,f_n(x_n)) = 0 \text{ for all } 1 \leq j \leq r\}. \]
\end{cor}

\begin{cor}
\label{AXKATZGENERALCOR1}
Let $P_1,\dots,P_r \in \F_q[t_1,\dots,t_n]$ be nonzero polynomials, and let $\I\subseteq\{1,\dots,n\}$ be a nonempty subset such that $\max_{1\le j\le r}\deg_{\I}(P_j)>0$. Let $f_1,\dots,f_n \in \F_q[t]$ and assume that $f_i$ is nonconstant for all $i\in\I$.  Suppose that $\sum_{j=1}^r \deg_{\I}(P_j) < \sum_{i\in\I} (1/\deg(f_i))$.
Then  
\[
q^{\left\lceil\left(\sum_{i\in\I} (1/\deg(f_i))-\sum_{j=1}^r \deg_{\I}(P_j)\right)/\max_{1\le j\le r}\deg_{\I}(P_j)\right\rceil}
\] 
divides $\# \{ (x_1,\dots,x_n) \in \F_q^n \mid P_j(f_1(x_1),\dots,f_n(x_n)) = 0 \text{ for all } 1 \leq j \leq r\}$.
\end{cor}

Let $m \in \Z^+$ and put $d = \gcd(m,q-1)$.  Then, for all $y \in \F_q$ we have 
\[ \# \{x \in \F_q \mid x^m = y\} = \# \{x \in \F_q \mid x^d = y\}, \]
so Corollary~\ref{AXKATZGENERALCOR1} implies the following result.
 
\begin{cor}
\label{AXKATZGENERALCOR}
Let $P_1,\dots,P_r \in \F_q[t_1,\dots,t_n]$ be nonzero polynomials, and let\linebreak $\I\subseteq\{1,\dots,n\}$ be a nonempty subset such that $\max_{1\le j\le r}\deg_{\I}(P_j)>0$. Let $f_1,\dots,f_n \in \F_q[t]$ with $f_i(t)=t^{m_i}$ for $m_i \in \Z^+$ for all $i\in\I$. For $i\in\I$, put $d_i=\gcd(m_i,q-1)$. Suppose that $\sum_{j=1}^r \deg_{\I}(P_j) < \sum_{i\in\I} (1/d_i)$.
Then  
\[
q^{\left\lceil\left(\sum_{i\in\I} (1/d_i)-\sum_{j=1}^r \deg_{\I}(P_j)\right)/\max_{1\le j\le r}\deg_{\I}(P_j)\right\rceil}
\] 
divides $\# \{ (x_1,\dots,x_n) \in \F_q^n \mid P_j(f_1(x_1),\dots,f_n(x_n)) = 0 \text{ for all } 1 \leq j \leq r\}$.
\end{cor}

Taking $\I=\{1,\dots,n\}$ and $m_1=\dots=m_n=1$ in Corollary~\ref{AXKATZGENERALCOR}, we recover Ax-Katz's Theorem~\ref{AXKATZTHM}. 
Taking $\I=\{1,\dots,n\}$, $r = 1$ and $\deg(P_1) = 1$, we recover Wan's Theorem~\ref{WAN88THM}. 
Taking $f_1(t)=\dots=f_n(t)=t$, we recover a result of Cao~\cite[Corollary~12]{Cao12}.

\subsection{$\bm{p}$-weight Degrees} 
We will give a further generalization of Theorem~\ref{AXKATZTHM} inspired by the work of Moreno-Moreno \cite{MM93}, \cite{MM95}.  For positive integers $M$ and $N\ge 2$, let $\sigma_N(M)$ be the sum of digits in the base $N$ representation of $M$. Let $\I$ be a nonempty subset of $\{1,\dots,n\}$. For a monomial $at_1^{m_1}\cdots t_n^{m_n}$ with $a\in\F_q^{\times}$, define 
$w_{p,\I}(at_1^{m_1}\cdots t_n^{m_n})\coloneqq \sum_{i\in\I}\sigma_p(m_i)$.  
For a polynomial $P=Q_1+\dots+Q_{\ell}$, where  $Q_1,\dots,Q_{\ell}\in\F_q[t_1,\dots,t_n]$ are monomials, we define the \textbf{$\bm{p}$-weight degree with respect to $\I$} as $w_{p,\I}(P)\coloneqq \max_{1\le k\le\ell}w_{p,\I}(Q_k)$. In the case of univariate polynomials we shall suppress the subscript $\I$.

\begin{theorem}
\label{MORENOGENERAL}
Let $P_1,\dots,P_r \in \F_q[t_1,\dots,t_n]$ be nonzero polynomials, and let\linebreak $\I\subseteq\{1,\dots,n\}$ be a nonempty subset such that $\max_{1\le j\le r}w_{p,\I}(P_j)>0$. Let $f_1,\dots,f_n \in \F_q[t]$ and assume that $f_i$ is nonconstant for all $i\in\I$.  Suppose that $\sum_{j=1}^r w_{p,\I}(P_j) \le \sum_{i\in\I}(1/w_p(f_i))$.
Then  
\[
p^{\left\lceil s\left(\sum_{i\in\I}(1/w_p(f_i))-\sum_{j=1}^r w_{p,\I}(P_j)\right)/\max_{1\le j\le r} w_{p,\I}(P_j)\right\rceil}
\] 
divides $\# \{ (x_1,\dots,x_n) \in \F_q^n \mid P_j(f_1(x_1),\dots,f_n(x_n)) = 0 \text{ for all } 1 \leq j \leq r\}$.
\end{theorem}

\begin{cor}
\label{MORENOGENERALCOR}
Let $P_1,\dots,P_r \in \F_q[t_1,\dots,t_n]$ be nonzero polynomials, and let\linebreak $\I\subseteq\{1,\dots,n\}$ be a nonempty subset such that $\max_{1\le j\le r}w_{p,\I}(P_j)>0$. Let $f_1,\dots,f_n \in \F_q[t]$ with $f_i(t)=t^{m_i}$ for $m_i \in \Z^+$ for all $i\in\I$. For $i\in\I$, put $d_i=\gcd(m_i,q-1)$. Suppose that 
$\sum_{j=1}^r w_{p,\I}(P_j) < \sum_{i\in\I}(1/\sigma_p(d_i))$. Then  
\[
p^{\left\lceil s\left(\sum_{i\in\I}(1/\sigma_p(d_i))-\sum_{j=1}^r w_{p,\I}(P_j)\right)/\max_{1\le j\le r} w_{p,\I}(P_j)\right\rceil}
\] 
divides
$\# \{ (x_1,\dots,x_n) \in \F_q^n \mid P_j(f_1(x_1),\dots,f_n(x_n)) = 0 \text{ for all } 1 \leq j \leq r\}$.
\end{cor}

Taking  $\I=\{1,\dots,n\}$ and $m_1=\dots=m_n=1$ in Corollary~\ref{MORENOGENERALCOR} recovers a result of Moreno and Moreno~\cite[Theorem~1]{MM93}, \cite[Theorem~1]{MM95}. Taking\linebreak $\I=\{1,\dots,n\}$, $r=1$ and $\deg(P_1)=1$ recovers a result of Moreno and Castro~\cite[Theorem~10]{MC08}. Taking $f_1(t)=\dots=f_n(t)=t$ recovers a result of Castro and Castro-Velez~\cite[Theorem~7]{CCV12}.

\subsection{Outline of the paper} The proof of Theorem~\ref{MAINTHM} is inspired by Ax's proof of the Chevalley-Warning Theorem~\cite{Ax64}.  We discovered Corollary~\ref{CORMAINTHMNEW} with\linebreak $\I=\{1,\dots,n\}$ first and proved it using a method inspired by Chevalley's proof of his theorem \cite{Chevalley35} and some polynomial method arguments.  The latter proof is an easy consequence of a new characterization of WSC polynomials (Theorem~\ref{BIGWSCTHM}) that has independent interest, so we will give both proofs here. The proof of Theorem~\ref{AXKATZGENERAL} is based on ideas and results from~\cite{Wan95} and \cite{ZC14}. 

We prove Theorem~\ref{BIGWSCTHM} in \S2. The proofs of Corollary~\ref{CORMAINTHMNEW},  Theorem~\ref{MAINTHM}, Theorem~\ref{AXKATZGENERAL} and Theorem~\ref{MORENOGENERAL} are presented in \S3, \S4, \S5 and \S6, respectively.

\section{A Characterization of WSC polynomials}

For a polynomial $f \in \F_q[t]$, let $f(\F_q)$ be its \textbf{value set}. For $y \in f(\F_q)$, let $e(y) \coloneqq \#\{x \in \F_q \mid f(x) = y\}$. Let $\varphi(t) = \prod_{y \in f(\F_q)} (t-y).$

\begin{theorem}
\label{BIGWSCTHM}
For a polynomial $f \in \F_q[t]$, the following are equivalent:
\begin{itemize}
\item[\textup{(a)}] The polynomial $f$ is WSC. 
\item[\textup{(b)}] The polynomial $f$ is weakly WSC, and for all $y \in f(\F_q)$, we have 
\begin{equation}
\label{BAOUEQ}
e(y) \varphi'(y) = C(f) \coloneqq \sum_{x \in \F_q} f(x)^{u(f)} \in \F_q^{\times}.
\end{equation} 
\item[\textup{(c)}] There is $C \in \F_q^{\times}$ such that $e(y) \varphi'(y) = C$ for all $y \in f(\F_q)$.
\end{itemize}
\end{theorem}
\begin{proof}
(a) $\implies$ (b): Being WSC, $f$ is weakly WSC.  The nonempty fibers of $E(f)$ partition $\F_q$, so
$\sum_{y \in f(\F_q)} e(y) = q = 0 \in \F_q$. Put $k = u(f) + 1$.  By definition of $u(f)$ and $C(f)$, we have
\begin{align*}
&\sum_{x \in \F_q} f(x)^{\delta} = 0, \text{ for all } 0 \leq \delta \leq k-2,\\
&\sum_{x \in \F_q} f(x)^{k-1} = C(f).
\end{align*}
Since $f$ is a WSC polynomial, we have $\# f(\F_q)= k$, so we may write $f(\F_q) = \{y_1,\dots,y_k\}$.  Then the 
above relations are equivalent to the linear system 
\begin{align*}
&\sum_{j=1}^{k} y_j^{\delta} e(y_j)= 0\text{ for all } 0 \leq \delta \leq k-2, \\
&\sum_{j=1}^{k} y_j^{k-1} e(y_j) = C(f).
\end{align*}
The matrix of coefficients has Vandermonde determinant
\[ \begin{vmatrix}
1 & 1 & \ldots & 1\\
y_1 & y_2 & \ldots & y_k\\
y_1^2 & y_2^2 & \ldots & y_k^2\\
\hdotsfor{4}\\
y_1^{k-1} & y_2^{k-1} & \ldots & y_k^{k-1}
\end{vmatrix}= \prod_{1 \leq i_1 < i_2 \leq k} (y_{i_2}-y_{i_1}) \neq 0.\]
Applying Cramer's rule, for all $1 \leq j \leq k$, we have
\begin{equation}
\label{eq1}
e(y_j)=\frac{(-1)^{j+k}C(f)\prod_{\substack{1\le i_1<i_2\le k\\ i_1\ne j, i_2\ne j}} (y_{i_2}-y_{i_1})}{\prod_{1\le i_1<i_2\le k} (y_{i_2}-y_{i_1})}.\qquad
\end{equation}
Since $\varphi'(y_j)=\prod_{\substack{i=1\\ i\ne j}}^{k}(y_j-y_i)$, we have
\begin{align*}
\prod_{1\le i_1<i_2\le k} (y_{i_2}-y_{i_1})&=\biggl(\prod_{\substack{1\le i_1<i_2\le k\\ i_1\ne j, i_2\ne j}} (y_{i_2}-y_{i_1})\biggr)\biggl(\prod_{i=j+1}^k (y_i-y_j)\biggr)\biggl(\prod_{i=1}^{j-1} (y_j-y_i)\biggr)\\
&=(-1)^{k-j}\varphi'(y_j)\prod_{\substack{1\le i_1<i_2\le k\\ i_1\ne j, i_2\ne j}} (y_{i_2}-y_{i_1}).
\end{align*}
Substituting this value into \eqref{eq1}, we obtain $e(y_j)=\frac{C(f)}{\varphi'(y_j)}$ and thus \eqref{BAOUEQ}. 

(b) $\implies$ (c) is immediate. 

(c) $\implies$ (a): By a result of Turnwald~\cite[Proposition~2.8]{Turnwald95} we have 
\begin{equation}
\label{TURNEQ1}q - u(f) - 1 = \deg(g),
\end{equation}   
where
\begin{align*}
g(t) &\coloneqq \biggl( \prod_{x \in \F_q} (t-f(x)) \biggr)' = \biggl( \prod_{y \in f(\F_q)}(t-y)^{e(y)} \biggr)'\\
&=\biggl(\prod_{y \in f(\F_q)}(t-y)^{e(y)-1}\biggr) \cdot \biggl( \sum_{y \in f(\F_q)} e(y) \prod_{z \in f(\F_q) \setminus \{y\}} (t-z) \biggr).
\end{align*}
The polynomial $h(t) \coloneqq \sum_{y \in f(\F_q)} e(y) \prod_{z \in f(\F_q) \setminus \{y\}} (t-z)$ has degree at most\linebreak $\#f(\F_q)-1$; 
moreover, for all $y \in f(\F_q)$, we have $h(y) = C$.  Thus $E(h)$ is constant and nonzero on a set of size larger than $\deg h$, so $\deg h = 0$ and \begin{equation}
\label{eq2}\deg(g) = \sum_{y \in f(\F_q)} (e(y)-1) = q - \# f(\F_q).
\end{equation}
Together \eqref{TURNEQ1} and \eqref{eq2} give $u(f) = \# f(\F_q) -1$, that is, $f$ is a WSC polynomial.
\end{proof}

\begin{remark}
As in Example~\ref{MONOEX}, let $f(t) = t^m$ be a monomial, and let\linebreak $d = \gcd(m,{q-1})$.  Then $e(0) = 1$ and for all $m$th powers $x \in \F_q^{\times}$, we have $e(x) = d$.  So Theorem~\ref{BIGWSCTHM} gives 
$\varphi'(0) = -1$ and $d\cdot \varphi'(x) = -1$ for all $m$th powers $x \in \F_q^{\times}$. 
In this case $\varphi(t) = t^{1+\frac{q-1}{d}}-t$, and differentiating and evaluating at 
$x$ also gives the result.  
\end{remark}

\section{Proof of Corollary~\ref{CORMAINTHMNEW}, following Chevalley}

\begin{lemma}
\label{AUXPROP}
Let $P_1,\dots,P_r \in \F_q[t_1,\dots,t_n]$ be nonzero polynomials. For $1 \leq i \leq n$, let $Y_i\subseteq \F_q$ be nonempty subsets,  and put $\varphi_i(t) \coloneqq \prod_{y \in Y_i}(t-y)$. Put $Y \coloneqq \prod_{i=1}^n Y_i$ and 
$V_Y \coloneqq \{ (y_1,\dots,y_n) \in Y \mid P_j(y_1,\dots,y_n) = 0 \text{ for all } 1 \leq j \leq r\}$.  
Suppose that $(q-1) \sum_{j=1}^r \deg(P_j) < \sum_{i=1}^n (\# Y_i -1)$. 
Then we have 
\[
\sum_{(y_1,\dots,y_n) \in V_Y} \frac{1}{\prod_{i=1}^n \varphi_i'(y_i)} = 0 \in \F_q.
\]  
\end{lemma}
\begin{proof}
This is \cite[Theorem~19(a)]{Clark14}.
\end{proof}

Let $P_1,\dots,P_r \in \F_q[t_1,\dots,t_n]$ be nonzero polynomials,  and let $\I\subseteq\{1,\dots,n\}$ be a nonempty subset. We assume without loss of generality that $\I=\{1,\dots,N\}$. Let $f_1,\dots,f_N \in \F_q[t]$ be WSC such that \eqref{MAINTHMEQ} holds, and let $f_{N+1},\ldots,$ $f_n\in\F_q[t]$.  Let
\[
X \coloneqq \{ (x_1,\dots,x_n) \in \F_q^n \mid P_j(f_1(x_1),\dots,f_n(x_n)) = 0 \text{ for all } 1 \leq j \leq r \}. 
\] 
For $a_{N+1},\dots,a_n\in\F_q$, let $X_{a_{N+1},\dots,a_n}$ denote the set
\begin{align*}
\{ (x_1,\dots,x_N) \in \F_q^N \mid\, & P_1(f_1(x_1),\dots,f_N(x_N),f_{N+1}(a_{N+1}),\dots,f_n(a_n))=0\\ &\text{for all } 1 \leq j \leq r\}. 
\end{align*}
Since $\# X=\sum_{(a_{N+1},\dots,a_n)\in\F_q^{n-N}} \# X_{a_{N+1},\dots,a_n}$, 
it suffices to show that $p$ divides $\# X_{a_{N+1},\dots,a_n}$ for any $(a_{N+1},\dots,a_n)\in\F_q^{n-N}$. Fix $(a_{N+1},\dots,a_n)\in\F_q^{n-N}$ and put 
\[
Q_j(t_1,\dots,t_N)\coloneqq P_j(t_1,\dots,t_N, f_{N+1}(a_{N+1}),\dots,f_n(a_n))\in\F_q[t_1,\dots,t_N]
\]
for all $1\le j\le r$. Then 
\[
X_{a_{N+1},\dots,a_n}\!=\{(x_1,\dots,x_N)\in\F_q^N\!\mid Q_j(f_1(x_1),\dots,f_N(x_N))=0 \text{ for all } 1\le j\le r\}.
\]
If $Q_1,\dots,Q_r$ are all identically zero, then $X_{a_{N+1},\dots,a_n}=\F_q^N$, which implies\linebreak $p\mid \# X_{a_{N+1},\dots,a_n}$. Now suppose that not all of $Q_1,\dots,Q_r$ are  identically zero. Without loss of generality we may assume that $Q_1,\dots,Q_M$ are nonzero polynomials and $Q_{M+1},\dots,Q_r$ are all identically zero. For all $1 \leq i \leq N$, let $Y_i \coloneqq f_i(\F_q)$. Let $Y\!\coloneqq \prod_{i=1}^N Y_i$ and $V_Y \coloneqq \{ (y_1,\dots,y_N) \in Y \mid Q_j(y_1,\dots,y_N) = 0 \text{ for all } 1 \leq j \leq M \}$. For $1 \leq i \leq N$ and $y_i \in Y_i$, let $\varphi_i(t) \coloneqq \prod_{y \in Y_i}(t-y)$ and $e_i(y_i) \coloneqq \# \{{x_i \in \F_q} \mid f_i(x_i) = y_i \}$. 
Each $(y_1,\dots,y_N) \in V_Y$ corresponds to $e_1(y_1) \cdots e_N(y_N)$ elements $(x_1,\dots,x_N) \in X_{a_{N+1},\dots,a_n}$ 
with $(f_1(x_1),\dots,f_N(x_N)) = (y_1,\dots,y_N)$. Since 
\[
(q-1)\sum_{j=1}^M \deg(Q_j)\le (q-1)\sum_{j=1}^r \deg_{\I}(P_j)< \sum_{i=1}^N (\# f_i(\F_q)-1),
\] 
Lemma~\ref{AUXPROP} and Theorem~\ref{BIGWSCTHM} give
\begin{align*} 
0 &= \sum_{(y_1,\dots,y_N) \in V_Y} \frac{1}{\prod_{i=1}^N \varphi_i'(y_i)}\\ 
&= \sum_{(x_1,\dots,x_N)\in X_{a_{N+1},\dots,a_n}} 
\frac{1}{e_1(f_1(x_1)) \cdots e_N(f_N(x_N))} \cdot\frac{1}{\prod_{i=1}^N \varphi_i'(f_i(x_i))} \\
&= \sum_{(x_1,\dots,x_N) \in X_{a_{N+1},\dots,a_n}} \frac{1}{\prod_{i=1}^N e_i(f_i(x_i)) \varphi_i'(f_i(x_i))} = 
\frac{\# X_{a_{N+1},\dots,a_n}}{\prod_{i=1}^N C(f_i)}.
\end{align*}
It follows that $p \mid \# X_{a_{N+1},\dots,a_n}$, completing the proof of Corollary~\ref{CORMAINTHMNEW}.

\section{Proof of Theorem~\ref{MAINTHM}, following Ax}

\begin{lemma}
\label{AMAZINGLEMMA}
Let $\I\subseteq\{1,\dots,n\}$ be a nonempty subset and let $f_1,\dots,f_n \in \F_q[t]$. Let $P \in \F_q[t_1,\dots,t_n] $ be a polynomial with $\deg_{\I}(P) < \sum_{i\in\I} u(f_i)$. Then 
\[
\sum_{(x_1,\dots,x_n) \in \F_q^n} P(f_1(x_1),\dots,f_n(x_n)) = 0.
\]
\end{lemma}
\begin{proof}
It suffices to consider a monomial $P(t_1,\dots,t_n) = t_1^{m_1} \cdots t_n^{m_n}$ 
with\linebreak $\sum_{i\in\I} m_i < \sum_{i\in\I} u(f_i)$.  Then there is $j\in\I$ such that 
$m_j < u(f_j)$, so\linebreak $\sum_{x_j \in \F_q} f_j(x_j)^{m_j} = 0$, and hence 
\[ \sum_{(x_1,\dots,x_n) \in \F_q^n} f_1(x_1)^{m_1} \cdots f_n(x_n)^{m_n} = \prod_{i=1}^n \biggl(\, \sum_{x_i \in \F_q} f_i(x_i)^{m_i} \biggr) = 0. \qedhere \]
\end{proof}

\smallskip
Let $P_1,\dots,P_r \in \F_q[t_1,\dots,t_n]$ be nonzero polynomials,  and let $f_1,\dots,f_n \in \F_q[t]$ be any polynomials such that \eqref{MAINTHMEQ1} 
holds.  Let
\[
X \coloneqq \{ (x_1,\dots,x_n) \in \F_q^n \mid P_j(f_1(x_1),\dots,f_n(x_n)) = 0  \text{ for all } 1 \leq j \leq r \}. 
\] 
We define Chevalley's polynomial $\chi(t_1,\dots,t_n) = \prod_{j=1}^r (1-P_j(t_1,\dots,t_n)^{q-1})$,
which has the property that for $(y_1,\dots,y_n) \in \F_q^n$, 
\[
\chi(y_1,\dots,y_n) = \begin{cases} 1 & \text{ if } P_1(y_1,\dots,y_n) = \dots = P_r(y_1,\dots,y_n) = 0, \\
0 & \text{ otherwise. } \end{cases}
\]  
Thus
\begin{align*} 
\# X &\equiv  \sum_{(x_1,\dots,x_n) \in \F_q^n} \chi(f_1(x_1),\dots,f_n(x_n)) \\
&\equiv \sum_{(x_1,\dots,x_n) \in \F_q^n} \prod_{j=1}^r (1-P_j(f_1(x_1),\dots,f_n(x_n))^{q-1})\\
&\equiv 
\sum_{(i_1,\dots,i_r) \in \{0,1\}^r} \sum_{(x_1,\dots,x_n) \in \F_q^n} \prod_{j=1}^r \left(-P_j(f_1(x_1),\dots,f_n(x_n))^{q-1} \right)^{i_j} \pmod{p}. 
\end{align*}
For all $(i_1,\dots,i_r) \in \{0,1\}^r$, we have 
\[ 
\deg_{\I} \biggl(\, \prod_{j=1}^r (-P_i^{q-1})^{i_j} \biggr) \leq (q-1) \sum_{j=1}^r \deg_{\I}(P_j) <  \sum_{i\in\I} u(f_i), 
\] 
so $p \mid \# X$ by Lemma~\ref{AMAZINGLEMMA}.

\section{Proof of Theorem~\ref{AXKATZGENERAL}}

Let $K_q$ be the unramified extension of $\Q_p$ of degree $s$. Let $v: K_q^\times \rightarrow \Z$ be the corresponding discrete valuation, and let $\mathcal{O}_{K_q}$ be the valuation ring.

\begin{lemma}
\label{NEWLEMMA1}
Let $x,y \in \mathcal{O}_{K_q}$ and let $n\in \Z^+$. If $p \mid (x-y)$, then $p^{n+1} \mid (x^{p^n}-y^{p^n})$.
\end{lemma}

\begin{proof} 
If $p \mid y$ then also $p \mid x$, so $p^{p^n} \mid (x^{p^n}-y^{p^n})$. Since $p^n\ge n+1$, the result follows in this case. So suppose that $p \nmid y$; thus $y$ is a unit in $\mathcal{O}_{K_q}$ and we may write $\frac{x}{y} = 1 + zp$ for some $z \in \mathcal{O}_{K_q}$. We have
\[ 
\left(\frac xy\right)^p-1 = \sum_{j=1}^p {\binom pj} (zp)^j. 
\]
For all $1 \leq j \leq p$, we have $p^2 z \mid {\binom pj} (zp)^j$, and so $v((x/y)^p-1) \geq 2 + v(z) = 1 + v((x/y)-1)$. By induction, for all $n \in \Z^+$ we have \[v((x/y)^{p^n}-1) \geq n + v((x/y)-1) \geq n+1. \] That is, $p^{n+1} \mid (x/y)^{p^n}-1$, hence $p^{n+1} \mid (x^{p^n}-y^{p^n})$.
\end{proof}

\begin{lemma}
\label{NEWLEMMA2}
Let $x \in \mathcal{O}_{K_q}$. Then
\[ 
x^{(q-1)q^n} \equiv \begin{cases} 
0 \pmod{q^n} & \text{if $p \mid x$}, \\ 
1 \pmod{q^n} & \text{if $p \nmid x$}. 
\end{cases}
\]
\end{lemma}

\begin{proof}
If $p \mid x$, then $p^{ns} \mid x^{(q-1)p^{ns}}$. If $p \nmid x$, then because $\mathcal{O}_{K_q}/p\mathcal{O}_{K_q} = \F_q$ we have $x^{q-1} \equiv 1 \pmod{p}$. Thus Lemma~\ref{NEWLEMMA1} gives $x^{(q-1)q^n}=x^{(q-1)p^{ns}}\equiv 1\pmod{p^{ns+1}}$, and the result follows.
\end{proof}
\noindent
Let $T_q=\{b\in K_q \mid b^q=b\}$ be the set of Teichm\"uller liftings of $\F_q$ in $K_q$.

\begin{lemma}
\label{ZANCAO}
Let $f\in\F_q[t]$ be a nonconstant polynomial and let $\tilde f$ denote the Teichm\"uller lifting of $f$ to $\mathcal{O}_{K_q}[t]$. Write $\tilde f(t)=\sum_{\ell=1}^R b_{\ell}t^{m_{\ell}}$, where $b_{\ell}\in \mathcal{O}_{K_q}\setminus\{0\}$  and $m_{\ell} \in \Z^{\geq 0}$. For $\delta \in \Z^{\geq 0}$, we have
\[
\sum_{x\in T_q} \tilde f(x)^{\delta}=q\tilde f(0)^{\delta}+(q-1)\sum_{\substack{\sum_{\ell=1}^R \delta_{\ell}=\delta\\ \sum_{\ell=1}^R m_{\ell}\delta_{\ell}\in(q-1)\Z^+}} \frac{\delta!}{\delta_1!\cdots \delta_R!}\,b_1^{\delta_1}\cdots b_R^{\delta_R}.
\]
\end{lemma}

\begin{proof}
This is a direct consequence of \cite[Proposition~2.2]{ZC14}.
\end{proof}

\begin{lemma}[\cite{MSCK04}, Proposition~11]
\label{MORENOSHUM}
For $L,M,N \in \Z^{\geq 0}$ with $N \geq 2$,
\begin{itemize}
\item[\textup{(a)}]
$\sigma_N(L)+\sigma_N(M) \ge \sigma_N(L+M)$;
\item[\textup{(b)}]
$\sigma_N(L)\sigma_N(M) \ge \sigma_N(LM)$;
\item[\textup{(c)}]
if $L\in(q-1)\Z^+$ then $\sigma_p(L)\ge\sigma_p(q-1)=s(p-1)$.
\end{itemize}
\end{lemma}

\smallskip
Let 
$X \coloneqq \{ (x_1,\dots,x_n) \in \F_q^n \mid P_j(f_1(x_1),\dots,f_n(x_n)) = 0  \text{ for all } 1 \leq j \leq r \}$.  We have 
$\# X =\# \{ (x_1,\dots,x_n) \in T_q^n \mid P_j(f_1(x_1),\dots,f_n(x_n)) \equiv 0 \pmod{p} \text{ for all }\linebreak 1 \leq j \leq r \}$,
where, by a slight abuse of notation, $P_j$ is used to denote the Teichm\"uller lifting of $P_j$ to $\mathcal{O}_{K_q}[t_1,\dots,t_n]$ and $f_i$ is used to denote the Teichm\"uller lifting of $f_i$ to $\mathcal{O}_{K_q}[t]$.  Applying Lemma~\ref{NEWLEMMA2}, we obtain
\begin{align*}
\# X &\equiv \sum_{(x_1,\dots,x_n) \in T_q^n} \prod_{j=1}^r (1-P_j(f_1(x_1),\dots,f_n(x_n))^{(q-1)q^n})\\
&\equiv 
\sum_{(i_1,\dots,i_r) \in \{0,1\}^r} \sum_{(x_1,\dots,x_n) \in T_q^n} \prod_{j=1}^r \left(-P_j(f_1(x_1),\dots,f_n(x_n))^{(q-1)q^n} \right)^{i_j} \!\!\!\!\!\!\!\pmod{q^n},
\end{align*}
and thus by induction on $r$ it suffices to show that 
\[
q^{\left\lceil\left(\sum_{i\in\I} (\omega(f_i)/(q-1))-\sum_{j=1}^r \deg_{\I}(P_j)\right)/\max_{1\le j\le r}\deg_{\I}(P_j)\right\rceil}
\]
divides
\[
A  \coloneqq \sum_{(x_1,\dots,x_n) \in T_q^n} \prod_{j=1}^r P_j(f_1(x_1),\dots,f_n(x_n))^{(q-1)q^n}.
\]
Write $P_j(t_1,\dots,t_n)= \sum_{k=1}^{L_j} a_{jk}t_1^{h_{1jk}}\cdots t_n^{h_{njk}}$, where $a_{jk}\in \mathcal{O}_{K_q}\setminus\{0\}$ and\linebreak $h_{ijk} \in \Z^{\geq 0}$, $1\le i\le n$, $1\le j\le r$, $1\le k\le L_j$. Using the multinomial theorem and interchanging the order of summation, we obtain
\begin{align}
\label{multinew}
A=\sum_{\substack{\beta_{11}+\dots+\beta_{1L_1}=(q-1)q^n\\ \hdots\\ \beta_{r1}+\dots+\beta_{rL_r}=(q-1)q^n}}&\Biggl[\biggl(\prod_{j=1}^r \frac{((q-1)q^n)!}{\beta_{j1}!\cdots \beta_{jL_j}!}a_{j1}^{\beta_{j1}}\cdots a_{jL_j}^{\beta_{jL_j}}\biggr)\Biggr.\\
&\Biggl.\times\biggl(\prod_{i=1}^n \sum_{x_i\in T_q} f_i(x_i)^{\sum_{j=1}^r \sum_{k=1}^{L_j}h_{ijk}\beta_{jk}}\biggr)\Biggr].\notag
\end{align}
Using Legendre's formula $\ord_p(m!)=(m-\sigma_p(m))/(p-1)$ (for a proof, see \cite[Lemma~6.39]{LN83}) and Lemma~\ref{MORENOSHUM}(c), we find
\begin{align}
\label{multicoef}
\ord_p \biggl(\frac{((q-1)q^n)!}{\beta_{j1}!\cdots \beta_{jL_j}!}\biggr)&=\frac 1{p-1}\biggl((q-1)q^n-s(p-1)-\sum_{k=1}^{L_j}(\beta_{jk}-\sigma_p(\beta_{jk}))\biggr)\\
&= \frac 1{p-1}\sum_{k=1}^{L_j} \sigma_p(\beta_{jk})-s.\notag
\end{align}
For $i\in\I$, write $f_i(t)=\sum_{\ell=1}^{R_i} b_{i\ell}t^{m_{i\ell}}$ with $b_{i\ell}\in \mathcal{O}_{K_q}\setminus\{0\}$  and $m_{i\ell} \in \Z^{\geq 0}$. By Lemma~\ref{ZANCAO},
\begin{equation}
\label{fi}
 \sum_{x_i\in T_q} f_i(x_i)^{\sum_{j=1}^r \sum_{k=1}^{L_j}h_{ijk}\beta_{jk}}=qf_i(0)^{\sum_{j=1}^r \sum_{k=1}^{L_j}h_{ijk}\beta_{jk}}+(q-1)B_i,
\end{equation}
where
\begin{equation}
\label{Bi}
B_i\coloneqq\sum_{\substack{\sum_{\ell=1}^{R_i}\gamma_{i\ell}=\sum_{j=1}^r \sum_{k=1}^{L_j}h_{ijk}\beta_{jk}\\
\sum_{\ell=1}^{R_i}m_{i\ell}\gamma_{i\ell}\in(q-1)\Z^+}} \frac{\bigl(\sum_{j=1}^r \sum_{k=1}^{L_j}h_{ijk}\beta_{jk}\bigr)!}{\gamma_{i1}!\cdots \gamma_{iR_i}!}\,b_{i1}^{\gamma_{i1}}\cdots b_{iR_i}^{\gamma_{iR_i}}.
\end{equation}
Without loss of generality we may assume that $\I=\{1,\dots,N\}$  and $\ord_p(B_i)<s$ for $1 \leq i \leq M$ and $\ord_p(B_i)\ge s$ for $M+1 \leq i \leq N$, with $0\le M\le N$. Now we examine the case $1\le i\le M$. In this case $B_i\ne 0$. We have
\begin{equation}
\label{qminus1}
\sum_{\ell=1}^{R_i}m_{i\ell}\gamma_{i\ell}\in(q-1)\Z^+\quad\text{for all}\quad 1\le i\le M.
\end{equation}
Again using Legendre's formula, we obtain
\begin{equation}
\label{multicoef1}
\ord_p\biggl(\frac{\bigl(\sum_{j=1}^r \sum_{k=1}^{L_j}h_{ijk}\beta_{jk}\bigr)!}{\gamma_{i1}!\cdots \gamma_{iR_i}!}\biggr)=\frac1{p-1}\biggl(\sum_{\ell=1}^{R_i}\sigma_p(\gamma_{i\ell})-\sigma_p\Bigl(\sum_{j=1}^r \sum_{k=1}^{L_j}h_{ijk}\beta_{jk}\Bigr)\biggr).
\end{equation}
Furthermore, $\sum_{\ell=1}^{R_i}m_{i\ell}\gamma_{i\ell}p^{\nu}\in(q-1)\Z^+$  for all $0 \leq \nu \leq s-1$ and $1\le i\le M$. Since $\gamma\equiv \sigma_q(\gamma)\pmod{q-1}$ and $\sigma_q(\gamma)=0$ iff $\gamma=0$, the above relations imply that for all $0 \leq \nu \leq s-1$ and $1\le i\le M$, we have $\sum_{\ell=1}^{R_i }m_{i\ell}\sigma_q(\gamma_{i\ell}p^{\nu})\in(q-1)\Z^+$, and thus  $\sum_{\ell=1}^{R_i }\sigma_q(\gamma_{i\ell}p^{\nu})\ge \omega(f_i)$. Summing over $i$, we obtain
\begin{equation}
\label{DEGREESUM}
\sum_{i=1}^M \omega(f_i)\le \sum_{i=1}^M \sum_{\ell=1}^{R_i} \sigma_q(\gamma_{i\ell}p^{\nu}).
\end{equation}
Using Lemma~\ref{MORENOSHUM}(a) and the fact that $\sum_{i=1}^M h_{ijk}\le\deg_{\I}(P_j)$, we see that
\begin{align}
\label{SIGMAQ}
\sum_{i=1}^M \sigma_q\Bigl(\sum_{j=1}^r\sum_{k=1}^{L_j} h_{ijk}\beta_{jk}p^{\nu}\Bigr)&\le \sum_{j=1}^r \sum_{k=1}^{L_j} \sigma_q(\beta_{jk}p^{\nu})\sum_{i=1}^M h_{ijk}\\&\le \sum_{j=1}^r \deg_{\I}(P_j)\sum_{k=1}^{L_j} \sigma_q(\beta_{jk}p^{\nu})\notag\\
&= (q-1)\sum_{j=1}^r \deg_{\I}(P_j)\notag\\
&+\sum_{j=1}^r \deg_{\I}(P_j)\biggl(\sum_{k=1}^{L_j} \sigma_q(\beta_{jk}p^{\nu})-(q-1)\biggr).\notag
\end{align}
Since $\sum_{k=1}^{L_j} \beta_{jk}p^{\nu}=(q-1)q^np^{\nu}$, $1\le j\le r$, we have  $\sum_{k=1}^{L_j}\sigma_q(\beta_{jk}p^{\nu})\in(q-1)\Z^+$. Hence
$\sum_{k=1}^{L_j}\sigma_q(\beta_{jk}p^{\nu})-(q-1)\ge 0$ for all $1\le j\le r$, and \eqref{SIGMAQ} implies
\begin{align*}
(q-1)\sum_{j=1}^r \deg_{\I}(P_j)&\ge \sum_{i=1}^M \sigma_q\Bigl(\sum_{j=1}^r\sum_{k=1}^{L_j} h_{ijk}\beta_{jk}p^{\nu}\Bigr)\\
&-\Bigl(\max_{1\le j\le r} \deg_{\I}(P_j)\Bigr)\sum_{j=1}^r \biggl(\sum_{k=1}^{L_j}\sigma_q(\beta_{jk}p^{\nu})-(q-1)\biggr).
\end{align*}
Subtracting this from \eqref{DEGREESUM} yields
\begin{align}
\label{SIGMAQ1}
\sum_{i=1}^M \omega(f_i)&-(q-1)\sum_{j=1}^r \deg_{\I}(P_j)\\
&\le \sum_{i=1}^M \biggl(\sum_{\ell=1}^{R_i} \sigma_q(\gamma_{i\ell}p^{\nu})-\sigma_q\Bigl(\sum_{j=1}^r\sum_{k=1}^{L_j} h_{ijk}\beta_{jk}p^{\nu}\Bigr)\biggr)\notag\\
&+\Bigl(\max_{1\le j\le r} \deg_{\I}(P_j)\Bigr)\sum_{j=1}^r \biggl(\sum_{k=1}^{L_j}\sigma_q(\beta_{jk}p^{\nu})-(q-1)\biggr).\notag
\end{align}
Since $\sum_{\ell=1}^{R_i} \gamma_{i\ell}p^{\nu}=\sum_{j=1}^r\sum_{k=1}^{L_j} h_{ijk}\beta_{jk}p^{\nu}$, appealing to Lemma~\ref{MORENOSHUM}(a) we deduce that
$\sum_{\ell=1}^{R_i} \sigma_q(\gamma_{i\ell}p^{\nu})-\sigma_q\Bigl(\sum_{j=1}^r\sum_{k=1}^{L_j} h_{ijk}\beta_{jk}p^{\nu}\Bigr)$ is a nonnegative multiple of $q-1$ for all $1\le i\le M$. Recalling that 
$\sum_{k=1}^{L_j}\sigma_q(\beta_{jk}p^{\nu})\equiv 0\pmod{q-1}$ and $\max_{1\le j\le r} \deg_{\I}(P_j)>0$, we see from \eqref{SIGMAQ1} that
\begin{align*}
\sum_{i=1}^M \sum_{\ell=1}^{R_i} \sigma_q(\gamma_{i\ell}p^{\nu})&-\sum_{i=1}^M\sigma_q\Bigl(\sum_{j=1}^r\sum_{k=1}^{L_j} h_{ijk}\beta_{jk}p^{\nu}\Bigr)+
\sum_{j=1}^r \sum_{k=1}^{L_j}\sigma_q(\beta_{jk}p^{\nu})-r(q-1)\\
&\ge (q-1)\left\lceil\frac{\sum_{i=1}^M (\omega(f_i)/(q-1))-\sum_{j=1}^r \deg_{\I}(P_j)}{\max_{1\le j\le r} \deg_{\I}(P_j)}\right\rceil.
\end{align*}
Summing over $\nu$ and using the fact (see \cite[p.~50]{Wan95}) that for any  $\gamma\ge 0$ we have
\[
\sum_{\nu=0}^{s-1} \sigma_q(\gamma p^{\nu})=\frac{q-1}{p-1}\sigma_p(\gamma),
\]
we deduce that
\begin{align*}
\frac 1{p-1}\biggl(\sum_{i=1}^M \sum_{\ell=1}^{R_i} \sigma_p(\gamma_{i\ell})&-\sum_{i=1}^M\sigma_p\Bigl(\sum_{j=1}^r\sum_{k=1}^{L_j} h_{ijk}\beta_{jk}\Bigr)+
\sum_{j=1}^r \sum_{k=1}^{L_j}\sigma_q(\beta_{jk})\biggr)-rs\\
&\ge s\left\lceil\frac{\sum_{i=1}^M (\omega(f_i)/(q-1))-\sum_{j=1}^r \deg_{\I}(P_j)}{\max_{1\le j\le r} \deg_{\I}(P_j)}\right\rceil.
\end{align*}
Combining the last inequality with \eqref{multinew}--\eqref{Bi} and \eqref{multicoef1} and recalling that\linebreak $\omega(f_i)\le q-1$ for all $1\le i\le N$, we conclude that
\begin{align*}
\ord_p(A)&\ge  s\left\lceil\frac{\sum_{i=1}^M (\omega(f_i)/(q-1))-\sum_{j=1}^r \deg_{\I}(P_j)}{\max_{1\le j\le r} \deg_{\I}(P_j)}\right\rceil+s(N-M)\\
&\ge s\left\lceil\frac{\sum_{i=1}^N (\omega(f_i)/(q-1))-\sum_{j=1}^r \deg_{\I}(P_j)}{\max_{1\le j\le r} \deg_{\I}(P_j)}\right\rceil,
\end{align*}
and so $q^{\left\lceil\left(\sum_{i\in\I} (\omega(f_i)/(q-1))-\sum_{j=1}^r \deg_{\I}(P_j)\right)/\max_{1\le j\le r}\deg_{\I}(P_j)\right\rceil}$ divides $A$.

\section{Proof of Theorem~\ref{MORENOGENERAL}}

We now give the proof of Theorem~\ref{MORENOGENERAL}, retaining the setup and notation of the previous section. Proceeding as in the proof of Theorem~\ref{AXKATZGENERAL} and applying Lemma~\ref{MORENOSHUM} to \eqref{qminus1}, we conclude that 
\begin{equation}
\label{spminus1}
\sum_{\ell=1}^{R_i}\sigma_p(m_{i\ell})\sigma_p(\gamma_{i\ell})\ge s(p-1)\quad\text{for all}\quad 1\le i\le M.
\end{equation}
Since $\sigma_p(m_{i\ell})\le w_p(f_i)$, we deduce that $\sum_{\ell=1}^{R_i} \sigma_p(\gamma_{i\ell})\ge s(p-1)/w_p(f_i)$ for all $1\le i\le M$. Summing this over $i=1,\dots,M$, we find that
\begin{equation}
\label{SIGMAP}
\sum_{i=1}^M \frac {s(p-1)}{w_p(f_i)}\le \sum_{i=1}^M \sum_{\ell=1}^{R_i} \sigma_p(\gamma_{i\ell}).
\end{equation}
Using Lemma~\ref{MORENOSHUM}(a, b) and the fact that $\sum_{i=1}^M \sigma_p(h_{ijk})\le w_{p,\I}(P_j)$, we see that
\begin{align}
\label{SIGMAP1}
\sum_{i=1}^M \sigma_p\Bigl(\sum_{j=1}^r\sum_{k=1}^{L_j} h_{ijk}\beta_{jk}\Bigr)&\le s(p-1)\sum_{j=1}^r w_{p,\I}(P_j)\\
&+\sum_{j=1}^r w_{p,\I}(P_j)\biggl(\sum_{k=1}^{L_j} \sigma_p(\beta_{jk})-s(p-1)\biggr).\notag
\end{align}
Since $\sum_{k=1}^{L_j} \beta_{jk}=(q-1)q^n$, parts (a) and (c) of Lemma~\ref{MORENOSHUM} yield
$\sum_{k=1}^{L_j}\sigma_p(\beta_{jk})\ge s(p-1)$, $j=1,\dots,r$. Now \eqref{SIGMAP1} implies
\begin{align*}
s(p-1)\sum_{j=1}^r w_{p,\I}(P_j)&\ge \sum_{i=1}^M \sigma_p\Bigl(\sum_{j=1}^r\sum_{k=1}^{L_j} h_{ijk}\beta_{jk}\Bigr)\\
&-\bigl(\max_{1\le j\le r} w_{p,\I}(P_j)\bigr)\sum_{j=1}^r \biggl(\sum_{k=1}^{L_j} \sigma_p(\beta_{jk})-s(p-1)\biggr).
\end{align*}
Subtracting this from \eqref{SIGMAP} gives
\begin{align}
\label{SIGMAP2}
\sum_{i=1}^M \frac {s(p-1)}{w_p(f_i)}&-s(p-1)\sum_{j=1}^r w_{p,\I}(P_j)\\
&\le\sum_{i=1}^M\biggl( \sum_{\ell=1}^{R_i} \sigma_p(\gamma_{i\ell})- \sigma_p\Bigl(\sum_{j=1}^r\sum_{k=1}^{L_j}
 h_{ijk}\beta_{jk}\Bigr)\biggr)\notag\\
&+\bigl(\max_{1\le j\le r} w_{p,\I}(P_j)\bigr)\sum_{j=1}^r \biggl(\sum_{k=1}^{L_j} \sigma_p(\beta_{jk})-s(p-1)\biggr).\notag
\end{align}
Since $\sum_{\ell=1}^{R_i} \gamma_{i\ell}=\sum_{j=1}^r\sum_{k=1}^{L_j} h_{ijk}\beta_{jk}$, appealing to Lemma~\ref{MORENOSHUM}(a) we deduce that
$\sum_{\ell=1}^{R_i} \sigma_q(\gamma_{i\ell})-\sigma_q\Bigl(\sum_{j=1}^r\sum_{k=1}^{L_j} h_{ijk}\beta_{jk}\Bigr)\ge 0$ for all $1\le i\le M$. Recalling that $\max_{1\le j\le r} w_{p,\I}(P_j)>0$, we conclude that
\begin{align*}
\frac 1{p-1}\biggl(\sum_{i=1}^M \sum_{\ell=1}^{R_i} \sigma_p(\gamma_{i\ell})&-\sum_{i=1}^M\sigma_p\Bigl(\sum_{j=1}^r\sum_{k=1}^{L_j} h_{ijk}\beta_{jk}\Bigr)+
\sum_{j=1}^r \sum_{k=1}^{L_j}\sigma_p(\beta_{jk})\biggr)-rs\\
&\ge s\cdot\frac{\sum_{i=1}^M  (1/w_p(f_i))-\sum_{j=1}^r w_{p,\I}(P_j)}{\max_{1\le j\le r} w_{p,\I}(P_j)}.
\end{align*}
Combining the last inequality with \eqref{multinew}--\eqref{Bi} and \eqref{multicoef1}, we see that
\begin{align*}
\ord_p(A)&\ge s\cdot\frac{\sum_{i=1}^M  (1/w_p(f_i))-\sum_{j=1}^r w_{p,\I}(P_j)}{\max_{1\le j\le r} w_{p,\I}(P_j)}+s(N-M)\\
&\ge s\cdot\frac{\sum_{i=1}^N  (1/w_p(f_i))-\sum_{j=1}^r w_{p,\I}(P_j)}{\max_{1\le j\le r} w_{p,\I}(P_j)}.
\end{align*}
Thus $p^{\left\lceil s\left(\sum_{i\in\I}  (1/w_p(f_i))-\sum_{j=1}^r w_{p,\I}(P_j)\right)/\max_{1\le j\le r} w_{p,\I}(P_j)\right\rceil}$ divides $A$.

\section*{Acknowledgments}
We thank G.~Ottinger, P.~Pollack and J.~R.~Schmitt for helpful conversations and the anonymous referee for the suggestion to make use of the work of D.~Q.~Wan.

\bibliographystyle{amsplain}

\begin{thebibliography}{MSCK04}

\bibitem[Ax64]{Ax64} 
J. Ax, \emph{Zeroes of polynomials over finite fields},
Amer. J. Math. \textbf{86} (1964), 255--261.

\bibitem[Ca12]{Cao12}
W.~Cao, \emph{A partial improvement of the Ax-Katz theorem}, J.~Number Theory \textbf{132} (2012), no.~4, 485--494.

\bibitem[CCV12]{CCV12}
F.~Castro and F.~N.~Castro-Velez, \emph{Improvement to Moreno--Moreno's theorems}, Finite Fields Appl. \textbf{18}~(2012), no.~6, 1207--1216.

\bibitem[Ch35]{Chevalley35} 
C. Chevalley, \emph{D\'emonstration d'une hypoth\`ese de M. Artin}, Abh. Math. Sem. Univ. Hamburg \textbf{11} (1935), 73--75.

\bibitem[Cl14]{Clark14} 
P. L. Clark, \emph{The Combinatorial Nullstellens\"atze Revisited}. Electron. J. Combin. 
\textbf{21} (2014), no.~4, Paper 4.15, 17 pp.

\bibitem[Jo71]{Joly71} 
J.-R. Joly, \emph{Nombre de solutions de certaines \'equations diagonales sur un corps fini},
C. R. Acad. Sci. Paris S\'er. A-B \textbf{272} (1971), A1549--A1552. 

\bibitem[Ka71]{Katz71} 
N. M. Katz, \emph{On a theorem of Ax}, Amer. J. Math. \textbf{93} (1971), 485--499.

\bibitem[LN83]{LN83}
R.~Lidl and H.~Niederreiter, \emph{Finite Fields}, Addison-Wesley, Reading, MA, 1983.

\bibitem[MC08]{MC08}
O.~Moreno and F.~N.~Castro, \emph{Optimal divisibility for certain diagonal equations over finite fields}, J.~Ramanujan Math. Soc. \textbf{23}~(2008), no.~1, 43--61.

\bibitem[MM93]{MM93}
O.~Moreno and C.~J.~Moreno, \emph{An elementary proof of a partial improvement to the the Ax-Katz theorem}. Applied algebra, algebraic algorithms and error-correcting codes (San Juan,
PR, 1993), Lecture Notes in Comput. Sci., vol. 673, Springer, Berlin, 1993, pp. 257--268.

\bibitem[MM95]{MM95}
O.~Moreno and C.~J.~Moreno, \emph{Improvements of the Chevalley-Warning and the Ax-Katz theorems}, Amer.~J. Math. \textbf{117}~(1995), no.~1, 241--244.

\bibitem[MSCK04]{MSCK04}
O.~Moreno, K.~W.~Shum, F.~N.~ Castro, and P.~V.~Kumar, \emph{Tight bounds for Chevalley--
Warning--Ax--Katz type estimates, with improved applications}, Proc. London
Math. Soc. (3) \textbf{88}~(2004), no.~3, 545--564.

\bibitem[Mo71]{Morlaye71} 
B. Morlaye, \emph{\'Equations diagonales non homog\`enes sur un corps fini}, 
C. R. Acad. Sci. Paris S\'er. A-B \textbf{272} (1971), A1545--A1548.

\bibitem[Tu95]{Turnwald95} 
G. Turnwald, \emph{A new criterion for permutation polynomials},
Finite Fields Appl. \textbf{1} (1995), no.~1, 64--82. 

\bibitem[Wan88]{Wan88} 
D.~Q. Wan, \emph{Zeros of diagonal equations over finite fields},
Proc. Amer. Math. Soc. \textbf{103} (1988), no.~4, 1049--1052. 

\bibitem[Wan95]{Wan95}
D.~Q.~Wan, \emph{A Chevalley-Warning approach to the $p$-adic estimates of character sums},
Proc. Amer. Math. Soc. \textbf{123}~(1995), no.~1, 45--54.

\bibitem[WSC93]{WSC93} 
D.~Q.~Wan, P.~J-.S.~Shiue and S.~S.~Chen, \emph{Value sets of polynomials over finite fields}, 
Proc. Amer. Math. Soc. \textbf{119} (1993), no.~3, 711--717.

\bibitem[War35]{Warning35} 
E. Warning, \emph{Bemerkung zur vorstehenden Arbeit von Herrn Chevalley},  Abh. Math. Sem. Hamburg \textbf{11} (1935), 76--83.

\bibitem[ZC14]{ZC14}
H.~Zan and W.~Cao, \emph{Powers of polynomials and bounds of value sets}, J.~Number Theory \textbf{143}, 286--292.

\end{thebibliography}

\end{document}